\documentclass[12pt]{article}
\usepackage{amsthm}
\usepackage{amsmath}
\usepackage{cite}
\usepackage{tikz}
\usetikzlibrary{arrows}
\usepackage{amssymb}
\usepackage{enumerate}
\usepackage[margin=1.3in]{geometry}
\usepackage{float}

\newtheorem{theorem}{Theorem}[section]

\newtheorem{corollary}[theorem]{Corollary}
\newtheorem{lemma}[theorem]{Lemma}
\newtheorem{definition}[theorem]{Definition}

\newtheorem{question}[theorem]{Question}

\newtheorem{mainthm}{Theorem}

\include{epsf}
\makeatletter
\newcommand*{\rom}[1]{\expandafter\@slowromancap\romannumeral #1@}
\makeatother

\begin{document}

\title{Hamiltonian torus actions and the unimodality of odd Betti numbers}

\author{Nicholas Lindsay}

\maketitle

\begin{abstract} This paper is dedicated to the question: Is the sequence of odd Betti numbers of a closed symplectic manifold with a Hamiltonian torus action unimodal? Recently, there was some progress on the question for the sequence of even Betti numbers by Cho-Kim and the author. The results of this paper give positive evidence in the case of odd Betti numbers, in dimensions $6,8$ and $10$ under progressively stronger symmetry assumptions. 
\end{abstract}

\section{Introduction}
The main motivating result of this article is the hard Lefschetz theorem. For a closed K\"{a}hler manifold $X$ define the Lefschetz operator as the cup product with the K\"{a}hler form.

\begin{theorem}  \label{HLT} \cite{Le}
For a closed K\"{a}hler manifold $X$ of complex dimension $n$ and integer $0<i<n$, the $i$'th power of the Lefschetz operator gives an isomorphism: $$L_{i}: H^i(X,\mathbb{R}) \rightarrow H^{2n-i}(X,\mathbb{R}).$$ 
\end{theorem}

For a closed symplectic manifold $M$, the Lefschetz operator is defined as the cup product with the symplectic form, and the above maps $L_{i}$ can be defined similarly. If $i$ is odd and $L_i$ is an isomorphism then $b_{i}(M)$ is even because $(\alpha,\beta) \mapsto \alpha \cup L_{i}\beta$ is a non-degenerate, skew-symmetric form. The first example of a closed symplectic $4$-manifold with $b_1(M)$ odd was the Kodaira-Thurston manifold \cite{Th}.  Examples of closed symplectic manifolds such that the underlying manifold is diffeomorphic to a K\"{a}hler manifold and $L_{2}$ is not an isomorphism were given in \cite[Theorem 1.3]{C1}. These examples have a non-trivial Hamiltonian $S^1$-action such that the reduced spaces are certain complex projective surfaces,  for instance a K3 surface. For cases where symmetry assumptions do imply that $L_{i}$ is an isomorphism, see \cite{CK3}.

Let $X$ be a closed K\"{a}hler manifold of complex dimension $n$ and $i$ an integer such that $i +2 \leq n$. Theorem \ref{HLT} implies that the Lefschetz operator is injective and therefore the inequality  $b_{i}(X) \leq b_{i+2}(X)$ holds. This collection of inequalities is referred to as \textit{the unimodality of Betti numbers}. The following problem was given by Tolman in \cite[Problem 4.3]{JHKLM}:
\begin{question} \label{question}
Let $(M,\omega)$ be a closed symplectic manifold of dimension $2n$ having a non-trivial Hamiltonian $S^1$-action. Are the Betti numbers of $M$ unimodal?
\end{question}

Let us mention some of the progress towards Question \ref{question} in the case of even Betti numbers. It follows from a result of Tolman and Weitsman \cite{TW} that if $(M,\omega)$ has a semi-free Hamiltonian $S^1$-action with isolated fixed points then $b_{i}(M) = {n \choose i}$, so the answer to Question \ref{question} is positive in this case. Under the assumption that the action has an index-increasing moment map a positive answer to Question  \ref{question} was given by Cho \cite{C}. 

In dimension $8$, progress towards Question \ref{question} in the case of even Betti numbers was made by Cho-Kim \cite[Theorem 1.2]{CK} and the author \cite[Theorem 5.1]{L2}. It turns out that extremal submanifolds, i.e. the fixed submanifolds of $M$ on which the Hamiltonian $H$ attains its minimum and maximum appear in the condition of  \cite[Theorem 5.1]{L2}. The extremal submanifolds are denoted $M_{\min},M_{\max} \subset M^{S^1} \subset M$ and are connected by a result of Atiyah and Guillemin-Sternberg \cite{A,GS}. 

This paper begins to address Question \ref{question} in the case where $k$ is odd.  The main results concern closed symplectic manifolds of dimensions $6,8$ and $10$ under progressively stronger symmetry assumptions. In the course of the proofs of these results some general observations which are useful for Question \ref{question} are made: it is shown in Lemma \ref{fourmin} that if there is an extremal component of dimension $0,4$ or $2n-2$ then $b_{1}(M) \leq b_{3}(M)$.
The proof is an application of the localization of Betti numbers theorem due to Kirwan \cite{Ki}, see Theorem \ref{locbet} for the statement.

The first main result of the article is as follows, it is proved in Theorem \ref{six}. 

 \begin{mainthm}\label{thm:mainA} Let $(M,\omega)$ be a closed, connected, symplectic $6$-manifold with a non-trivial Hamiltonian $S^1$-action. Then $b_{1}(M) \leq b_{3}(M)$.
\end{mainthm} The proof is an application of the localization of Betti numbers theorem due to Kirwan (see Theorem \ref{locbet}) and previous work of the author and Panov \cite{LP1} (see Theorem \ref{sixfixed}). We give a brief overview of the argument.

By Lemma \ref{fourmin} it is only necessary to prove the statement for the case $\dim(M_{\min}) = \dim(M_{\max}) = 2$.  By Theorem \ref{locbet}, $b_{1}(M_{\min}) = b_{1}(M_{\max}) = b_{1}(M)$ hence $M_{\min},M_{\max}$ have the same genus. If  $g(M_{\min}) = 0$ then $b_{1}(M) = 0$ by Theorem \ref{locbet} so the desired statement is proven.  If $g(M_{\min}) >0$  then Theorem \ref{sixfixed} ensures the existence of a non-extremal fixed component $F$ of dimension $2$ with  $g(F) \geq g(M_{\min})=g(M_{\max})$, therefore by Theorem \ref{locbet} $b_1(M) \leq b_{3}(M)$.

The second main result of the article is as follows, it is proved in Theorem \ref{eightmain}.
 \begin{mainthm}\label{thm:mainB} Let $(M,\omega)$ be a closed, connected, symplectic $8$-manifold with an effective Hamiltonian $S^1$-action, such that each non-zero weight along $M_{\min}$ and $M_{\max}$ has modulus greater than $1$. Then $b_{1}(M) \leq b_{3}(M)$.
\end{mainthm}
Using Lemma \ref{fourmin} and  Theorem \ref{sixfixed}, we reduce the proof to the situation where $\dim(M_{\min}) = \dim(M_{\max}) = 2$ and there are $3$, $4$-dimensional isotropy manifolds containing $M_{\min}$. Moreover, since the maximum of each of these isotropy $4$-manifolds is a surface with genus equal to that of $M_{\min}$, it must hold that all three of these isotropy $4$-manifolds contain $M_{\max}$. Then,  \cite[Lemma 2.6]{T} implies various congruences between the weights at $M_{\min}$ and $M_{\max}$. We show using arithmetic arguments in Lemma \ref{numberprop} that such a system of congruences cannot hold.

The final main result of the article is as follows, it is proved in Theorem \ref{tenmain}.

 \begin{mainthm}\label{thm:mainC} Let $(M,\omega)$ be a closed, connected, symplectic $10$-manifold with an effective Hamiltonian $T^2$-action. Then $b_{1}(M) \leq b_{3}(M)$.
\end{mainthm}

Firstly, in Subsection \ref{ttwo} it is shown that the moment map preimages of vertices and edges of the moment polygon are extremal fixed submanifolds for certain subcircle actions of the $T^2$-action. By the aforementioned Lemma \ref{fourmin} if there is a Hamiltonian $S^1$-action having an extremal component of dimension $0,4$ or $8$ then the desired inequality holds. Therefore it is sufficient to deal with the case that the moment map preimages of vertices and edges are either $2$ or $6$ dimensional, and since the moment map preimages of edges have greater dimension than the moment map preimages of vertices, the moment map preimages of edges have dimension $6$. Since, as is recalled in Lemma \ref{edgeiso} of the preliminaries, such moment map preimages themselves admit a non-trivial Hamiltonian $S^1$-action, applying Theorem \ref{thm:mainA} shows that these submanifolds satisfy the desired Betti number inequality. Finally, pulling this inequality back to $M$ via the localization of Betti numbers Theorem \ref{locbet} proves the theorem. In the last subsection a result in arbitrary dimensions is proven.

\textbf{Acknowledgements.} I would like to thank the referees for helpful comments and the much shortened proof of Lemma 3.4. The author is funded by the position: Postdoc in complex and symplectic geometry in memory of Paolo de Bartolomeis. 
 \section{Preliminaries}

 \subsection{Torus actions}
In this section, some foundational preliminaries for smooth torus actions are presented. Let $M$ be a closed, connected, smooth manifold with a smooth effective action of a compact torus $T^{m}$. For a subgroup $K$ of $T^m$, set $$M^{K} = \{p \in M | k.p=p \;\; \forall k \in K\}.$$

$M^{K}$ is called the isotropy submanifold associated to $K$. $M^{K}$ may be described alternatively as the subset of $M$ consisting of points whose stabilizer contains $K$.

\begin{lemma} \label{isoprop}
For a subgroup $K \subset T^{m}$, it holds that: \begin{enumerate} \item Every connected component of $M^{K}$ is $T^{m}$-invariant.
\item Every connected component of $M^{K}$ is a smooth submanifold.
\end{enumerate}
\end{lemma}
\begin{proof}
1. The fact that $M^{K}$ is $T^m$-invariant follows directly from the fact that $T^{m}$ is abelian. Let $p \in M^{K}$, then for any $g \in T^m$, $k \in K$, $k.g.p = g.k.p = g.p$, that is $g.p \in M^{K}$. Since orbits of the $T^{m}$-action are connected, invariance of the individual connected components follows.

2. By averaging, we may pick a $T^{m}$-invariant metric on $M$. Then, the exponential around any point is a local $T^{m}$-equivariant diffeomorphism. When $p \in M^{K}$, $M^{K}$ is locally the image of the subspace $V_{K} \subset T_{p}(M)$, where $V_{K}$ is the intersection of the $1$-eigenspaces of all the transformations of $T_{p}(M)$ of the form $k_{*}$, $k \in K$. Since it is the image of a linear subspace by a local diffeomorphism it is a smooth submanifold. \end{proof}

\subsection{Hamiltonian torus actions} \label{hamtoractions}
Let $w = (w_{1},\ldots,w_m) \in \mathbb{Z}^{m}$. The complex representation with weight $w$ is: $$(z_1,\ldots,z_m) .z= z_{1}^{w_1} \ldots z_{m}^{w_m} z . $$ It is known that these are the irreducible representations of $T^{m}$. For an almost complex $2n$-manifold with an almost complex $T^{m}$-action, and a fixed point $p \in M^{T^m}$, note that $T_{p}M$ has the structure of a $T^{m}$-representation. Therefore, it decomposes into $n$ irreducible representations, the associated weights are called the weights of the action at $p$. For each fixed component $F \subset M^{T^m}$ they form a multiset of cardinality $n$ consisting of elements of $\mathbb{Z}^m$.  

\begin{definition}
Let $(M,\omega)$ be a closed, symplectic manifold, with a smooth $T^m$-action by symplectomorphims. Then, the action is called Hamiltonian if there is a smooth $T^m$-invariant function $\mu: M \rightarrow \mathfrak{t}^*$ such that $$\omega (X_t, \cdot) = -d(\mu(t)),$$ for all $t \in  \mathfrak{t}$.
\end{definition}

We will need the following fundamental localization formula for the Betti numbers due to Kirwan \cite{Ki}. Let $(M,\omega)$ be a closed symplectic manifold having a Hamiltonian $S^1$-action. For $p \in M^{S^1}$ denote by $\lambda_{p}$  the number of negative weights of the action on $T_{p}M$ counted with multiplicity. For a fixed component $F$, $\lambda_F$ is well-defined because the weights of points in given fixed component are equal (see \cite[Section 5]{MS}).

 \begin{theorem} \label{locbet} \cite[Page 34]{Ki}
 Let $(M,\omega)$ be a closed, connected, symplectic manifold having a Hamiltonian circle action. Then for each $i$, $$ b_{i}(M) = \sum_{F \subset M^{S^1}} b_{i - 2\lambda_{F}(M)} (F) ,$$ where the sum runs over connected components of the fixed point set. 
 \end{theorem}
A more self-contained proof of Theorem \ref{locbet} was given in \cite[Page 7]{T}. The following lemma about isotropy submanifolds will be needed.
\begin{lemma} \label{isotropy}
Let $(M,\omega)$ be a closed, connected, symplectic manifold with a Hamiltonian $T^m$-action and moment map $\mu$. Let $K$ be a closed subgroup of $T^m$ and $N$ an isotropy submanifold i.e. a connected component of $M^{K}$. Then $N$ is a symplectic submanifold and the restricted action of $T^m$ induces a Hamiltonian $T^{m}/K$-action with moment map $\mu|_{N}$.
\end{lemma}

Note that when $m=1$ then by a continuity argument the only subgroups $K \subset S^1$ for which can satisfy $M^{K} \neq M^{S^1}$ are cyclic subgroups $\mathbb{Z}_{d} \subset S^1$ generated by $e^{\frac{2\pi i}{d}}$. In this case the action of $S^1/K \cong S^1$ is an effective Hamiltonian $S^1$-action on $M^{K}$. Another key fact which will be used repeatedly is the following.

\begin{lemma} \label{fixedintersection}
Let $N$ be an isotropy submanifold then $N^{T^m/K} = N \cap M^{T^m}$.
\end{lemma}

Generally speaking, this allows arguments which leverage low-dimensional results to prove higher dimensional ones, because by Lemma \ref{isotropy} isotropy submanifolds provide natural symplectic submanifolds which themselves have Hamiltonian torus actions. Moreover, by Lemma \ref{fixedintersection} the fixed point sets of isotropy submanifolds give direct information about the fixed point set of $M$. 

Another fact which is used repeatedly throughout the article, referred to as reversing the circle action, is the following.
\begin{lemma} \label{reverse}
Suppose that $(M,\omega)$ is a closed, connected, symplectic manifold with a non-trivial Hamiltonian $S^1$-action,  and $N = M_{\max}$. Then there exists a Hamiltonian $S^1$-action on $M$ with $N=M_{\min}$.
\end{lemma}
\begin{proof}
The desired circle action has generating vector field $-X$ and Hamiltonian $-H$.
One may check that by linearity $(M,\omega, -H, -X)$ satisfies the Hamiltonian equation. \end{proof}

The Atiyah-Guillemin-Sternberg convexity theorem will be used, we state it for convenience.

\begin{theorem} \cite{A,GS}
Let $(M,\omega)$ be a closed, connected, symplectic manifold with an effective, Hamiltonian $T^m$-action. Then $\mu(M)$ is equal to the convex hull of $\mu(M^{T^m})$, in particular it is a convex polytope. 
\end{theorem}

One well-known and easily-verified fact from convex geometry that we will use without further reference is that for a finite, convex polytope which is the convex hull of a finite set $A$ the set of vertices is contained in $A$. Therefore every vertex of the moment polytope is contained in the moment map image of the fixed point set.

Finally, we recall a result of the author and Panov  \cite[Theorem 9.10]{LP1} that will be applied in Section \ref{maincircle}.

\begin{theorem} \label{sixfixed}
Suppose that $(M,\omega)$ is a closed, connected, symplectic $6$-manifold with a Hamiltonian $S^1$-action such that $\dim(M_{\min})=\dim(M_{\max}) = 2$, and the genus of $M_{\min}$ is positive. Then, there exists a non-extremal fixed surface $F$ and a map $f: F \rightarrow M_{\min}$ such that $\deg(f) \in \{1,2\}$.
\end{theorem}
The important consequence for this article is that the existence of the map $f$ implies that $b_{1}(F) \geq b_{1}(M_{\min})$. In fact, such an inequality holds for any Betti number and any map of positive degree between closed, connected, orientable manifolds \cite[Lemma 1.2]{R}.
\subsection{Hamiltonian $T^2$-actions} \label{ttwo}
In this subsection of the preliminaries, we prove some facts about the preimages of vertices and edges of the moment polygon of a Hamiltonian $T^2$-action.

Before giving the first lemma, we recall some basic preliminaries about lattices and recall an elementary fact. Let $m>0$ be an integer and $V$ be a rank $m$ real vector space, a lattice $l \subset V$ is a subgroup isomorphic to $\mathbb{Z}^m$ which is a spanning set of $V$. Recall that the dual lattice $l^* \subset V^*$ is defined by $l^* := \{f \in V^* |  f(l) \subset \mathbb{Z}\}$. The following lemma about rank $2$ lattices, has been separated out from the proof of Lemma \ref{edgeiso} to avoid it being obscured by notations.

\begin{lemma} \label{latticelemma}
Let $V$ be a rank $2$ real vector space and $l \subset V$ a lattice. For each $u \in l^* \setminus \{0\}$,  $\ker(u)$ is generated by some $v \in l \setminus \{0\}$.
\end{lemma}
\begin{proof}
Let $e_1,e_2$ be an integral basis of $l$, since $l$ is a lattice $e_1,e_2$ is a basis of $V$. If $e_1^*,e_2^*$ is the corresponding dual basis of $V$, note that  $e_1^*,e_2^*$ is an integral basis of $l^*$. Writing $u = pe_1^*+qe_2^*$ for some $p,q \in \mathbb{Z}$, $v = -qe_1+pe_2$ generates $\ker(u) \subset V$. 
\end{proof}
These preliminaries will be applied to the following lattice. Let $T^m = S^1 \times \ldots \times S^1$ be a compact torus. The Lie algebra $\mathfrak{t}$ is defined to be the tangent space of $T^m$ at the identity. The weight lattice $l \subset \mathfrak{t}$ is defined as the kernel of the exponential $\exp: \mathfrak{t} \rightarrow T$. 

\begin{lemma} \label{edgeiso}
Let $(M,\omega)$ be a  closed, connected, symplectic $2n$-manifold with a Hamiltonian $T^2$-action, $E$ an edge of the moment polygon $\mu(M)$. Then there exists a linear subcircle $S_{E} \subset T^2$ such that the following properties hold: \begin{enumerate} 

\item $\mu^{-1}(E)$ is a connected component of $M^{S_{E}}$ and is $T^2$-invariant.

\item There is a Hamiltonian $S^1$-action on $(M,\omega)$, with minimal fixed submanifold equal to $\mu^{-1}(E)$.

\item $\mu^{-1}(E)$ admits a non-trivial Hamiltonian $S^1$-action.  \end{enumerate}
\end{lemma}

\begin{proof}
Let $l \subset \mathfrak{t}$ be the weight lattice,  and  $l^* \subset \mathfrak{t}^*$ the dual lattice. Let $v$ be a vertex of $E$. By the Atiyah-Guillemin-Sternberg convexity theorem  $\mu^{-1}(v)$ contains a fixed point say $p \in M^{T^2}$. By the equivariant symplectic neighborhood theorem  \cite[Theorem A.1]{Ka}, there exists a neighborhood of $p$ such that the action is equivariantly symplectomorphic to the action on the standard symplectic vector space given by a linear embedding $s: T^2 \rightarrow T^{n}$, let $\pi_i : T^{n} \rightarrow S^1$ be the projection to the $i$'th coordinate. Then $w_i := (\pi_i \circ s)_{*} \in l^*$ is the $i'th$ weight at $p$, moreover in these local coordinates $$\mu(z_{1},\ldots,z_{n}) = w_i |z_i|^2 +c.$$ Moreover, by adding a constant to $\mu$ we may assume $c=0$.

It follows that  locally the image of $\mu$ is the cone over $\{ w_i | i=1, \ldots, n\}$, i.e. vectors of the form $\sum x_i w_i$, such that $x_i \geq 0$. Since $E$ is on the boundary of the polygon $\mu(M)$, it follows that  for at least one $i$, the direction of $E$ centered at $\mu(p)$ is $w_i \in l^*$, fix $w_i$ to be such a weight. By Lemma \ref{latticelemma} the kernel of the map $w_i : \mathfrak{t} \rightarrow \mathbb{R}$ is generated by an element $u \in l \setminus \{0\}$. Let $t_{e}$ be a primitive element such that $kt_{e} = u$, $k \in \mathbb{Z}$. Denote by $X_e $ the corresponding invariant vector field on $M$. Since $t_e \in l$ is primitive,  $X_e$ generates the action of the corresponding linear subgroup $ S_{E} \subset T^2$. 

By applying the Hamiltonian equation of the $T^2$-action, $d (\langle \mu,t_{e} \rangle) = \omega(X_{e},\cdot)$, so the Hamiltonian of the action of the linear subgroup $S_{E}$ is $H(p) = \langle \mu(p), t_{e} \rangle$.  Moreover, since by definition $\langle w_i, t_{e} \rangle = 0$, it holds that: $$H^{-1}(0) =  \{ p \in M | \langle \mu(p), t_{e} \rangle = 0   \} = \{ p \in M | \mu(p)  = kw_e, k \in \mathbb{R}   \}  =  \mu^{-1}(E).$$

Note that if a linear function is constant on the edge of a convex polygon, then the edge must be an extremum for the restriction of the function to the polygon. Therefore, $\mu^{-1}(E)$ is an extremum of $H$ and therefore is fixed pointwise by $S_{E}$, i.e. $\mu^{-1}(E) \subset M^{S_{E}}$. As  $\mu^{-1}(E)$ is subset of $M^{S_{E}}$ consisting of a level set of $H$ it is connected, moreover it is a connected component of  $M^{S_{E}}$  because any other point in $M$ has a different value of $H$, therefore cannot be connected by a path in  $M^{S_{E}}$.

Finally to exhibit a non-trivial Hamiltonian $S^{1}$-action on $\mu^{-1}(E)$ let $S'$ be any linear circle subgroup whose tangent space at the identity is linearly independent from that of $S_{E}$, since by Lemma \ref{isoprop} each connected component of $M^{S_{E}}$ is $T^2$-invariant, $\mu^{-1}(E)$ is in particular $S'$-invariant. Finally, since the action of $S'$ has a non-zero weight at $p$ so it is non-trivial.  
\end{proof}
The following lemma is similar but simpler, therefore its proof has been abbreviated significantly. Several of the technical details can be found in the proof of Lemma \ref{edgeiso}.
\begin{lemma} \label{vertexfixed}

Let $(M,\omega)$ be closed, connected, symplectic $2n$-manifold with an effective Hamiltonian $T^2$-action, $v$ a vertex of the moment polygon $\mu(M)$. Then there is a Hamiltonian $S^1$-action on $(M,\omega)$ associated to linear subcircle $S \subset T^2$, which has extremal fixed submanifold equal to $\mu^{-1}(v)$.
\end{lemma}
\begin{proof}
Let $E_1,E_{2}$ be the edges of $\mu(M)$ containing $v$. By the Atiyah-Guillemin-Sternberg convexity theorem, there exists some $p \in \mu^{-1}(v)$ that is fixed by $T^2$.  Consider the fixed point representation and the local form of the action on a neighborhood of $p$ as in the beginning of the proof of Lemma \ref{edgeiso}. As in the previous proof, the image locally, up to translations is a cone, over the weights of the action, therefore two of the weights $w_1,w_2 \in l^*$ correspond to the edges up to a constant. Taking $w' \in l^*$ to be any primitive element in the interior of the moment polygon (for example $w_1+w_2$) and taking $t' \in l$ a primitive element in the kernel provided by Lemma \ref{latticelemma}, the circle action associated to $t'$ has the desired property.
 \end{proof}

\section{Proof of the main results for $S^1$-actions} \label{maincircle}
The section begins by showing that for a closed symplectic manifold with a Hamiltonian $S^1$-action, certain values of the dimension of the  extremal submanifolds immediately imply unimodality inequalities in the Betti numbers  of $M$. The first main goal is to prove Lemma \ref{fourmin} which is central to all the following proofs of the paper and then prove the first two main results of the paper. 

Let $(M,\omega)$ be a closed symplectic manifold with a non-trivial Hamiltonian $S^1$-action with Hamiltonian $H: M \rightarrow \mathbb{R}$. Then $$ M^{S^1} := \{p \in M | z.p=p \; \forall z \in S^1\},$$ is a union of symplectic submanifolds and a connected component of $M^{S^1}$ is called a fixed component.  It follows from the Hamiltonian equation that $H$ is constant on fixed components. A fixed component $F \subset M^{S^1}$ is called extremal if it is equal to one of the subsets on which $H$ attains its minimum or maximum, which are denoted $M_{\min},M_{\max} \subset M^{S^1} \subset M$. In the course of the proof of the convexity theorem, Atiyah-Guillemin-Sternberg showed that $M_{\min},M_{\max}$ are fixed components. In Subsection \ref{hamtoractions} a key quantity $\lambda_F$ for fixed components $F$ was defined.  Since $\lambda_{F}$ is half the Morse-Bott index of $H$ along $F$ \cite[Page 7]{T}, the following fact follows: For a fixed component $F \subset M^{S^1}$ $$\lambda_F=0 \iff F = M_{\min}.$$  This will be used in the proofs throughout this section.

\begin{lemma}\label{unimod}
Suppose that $(M,\omega)$ is a $2n$-dimensional closed, connected, symplectic manifold with a non-trivial Hamiltonian $S^1$-action, and has an extremal fixed component $F_{0}$ of dimension $2k$. Then the inequalities $b_{1}(M) \leq b_{2k-1}(M)$ and  $b_{1}(M) \leq b_{2(n-k)+1}(M)$ hold.
\end{lemma}
\begin{proof}
By possibly reversing the circle action by  Lemma \ref{reverse}, we may assume that $F_0 = M_{\min}$. By Theorem \ref{locbet}, $$b_{1}(M) = \sum_{F \subset M^{S^1}} b_{1-2\lambda_F} (F)= b_{1}(M_{\min}) .$$

By applying Poincar\'{e} duality to $M_{\min}$ $b_{1}(M_{\min}) =  b_{2k-1}(M_{\min})$, therefore by  Theorem \ref{locbet} $$ b_{2k-1}(M) =  \sum_{F \subset M^{S^1}} b_{2k-1 -2\lambda_F} (F) = b_{2k-1}(M_{\min}) + \sum_{F \subset M^{S^1}, \; F \neq M_{\min}} b_{2k-1 -2\lambda_F} (F).$$ Therefore, as was shown above $b_{2k-1}(M_{\min}) = b_{1}(M)$, it follows that $b_{1}(M) \leq b_{2k-1}(M)$. By Poincar\'{e} duality applied to $M$, $b_{2k-1}(M) = b_{2n-(2k-1)}(M)$ proving the second inequality.
\end{proof}

Next we prove a lemma that will be central to the proofs of all the main results of the article.

\begin{lemma} \label{fourmin}
Suppose that $(M,\omega)$ is a $2n$-dimensional, closed, connected, symplectic manifold with a non-trivial Hamiltonian $S^1$-action, and has an extremal fixed component $F_{0}$ of dimension $0$, $4$ or $2n-2$. Then $b_{1}(M) \leq b_{3}(M)$.
\end{lemma}
\begin{proof}
By possibly reversing the circle action by  Lemma \ref{reverse} we can assume $F_0=M_{\min}$. Firstly, if $F_{0}$ is isolated, by Theorem \ref{locbet}  $$b_{1}(M) =  \sum_{F \subset M^{S^1}} b_{1-2\lambda_F} (F) = b_{1}(M_{\min}) = 0$$ so the statement holds. Next, assume $\dim(F_0) = 4$ then the result holds by the first inequality of Lemma \ref{unimod}. When $\dim(F_0) = 2n-2$, the conclusion is a consequence of the second inequality of Lemma \ref{unimod}.
\end{proof}

\begin{theorem} \label{six}
Let $(M,\omega)$ be a closed, connected, symplectic $6$-manifold with a non-trivial Hamiltonian $S^1$-action. Then $b_{1}(M) \leq b_{3}(M)$.
\end{theorem}
\begin{proof}   By Lemma \ref{fourmin} it is necessary only to prove the statement for the case $\dim(M_{\min}) = \dim(M_{\max}) = 2$.  By Theorem \ref{locbet}, $b_{1}(M_{\min}) = b_{1}(M_{\max}) = b_{1}(M)$ hence $M_{\min}$ and $M_{\max}$ have the same genus, which we denote $g(M_{\min}) =g(M_{\max})$. If  $g(M_{\min}) = 0$ then $b_{1}(M) = 0$ by Theorem \ref{locbet} so the desired statement is proved. If $g(M_{\min})>0$, by Theorem \ref{sixfixed} there exists a non-extremal fixed surface $F$ with a map $f: F \rightarrow M_{\min}$ of degree  $\deg(f) \in \{1,2\}$. In particular, $f_{*}:  H_{1}(F,\mathbb{R}) \rightarrow H_{1}(M_{\min},\mathbb{R})$ is surjective and it follows that $b_{1}(F) \geq b_{1}(M_{\min}) = b_{1}(M_{\max}) = b_{1}(M)$. Since $F$ is non-extremal and $2$-dimensional by a dimension count $\lambda_{F}=1$. By Theorem \ref{locbet} it holds that  $b_{3}(M) \geq b_{1}(F)$. 
\end{proof}

In Theorem \ref{eightmain} below, we prove an analogous result in dimension $8$ under an additional assumption on the action. Firstly, we prove an arithmetic lemma which will be used in the proof. Let $k>1$ be an integer and $\pi : \mathbb{Z} \rightarrow  \mathbb{Z}/k\mathbb{Z}$ the quotient map. Let $A_1,A_2 \subset \mathbb{Z}$ be finite subsets, with indicator functions $m_{A_1},m_{A_{2}}: \mathbb{Z} \rightarrow \{0,1\}$. Let $\tilde{A}_1 ,\tilde{A}_2$ be the multisets consisting of elements of  $\mathbb{Z}/k\mathbb{Z}$ with indicator functions $m_{\tilde{A}_i} :   \mathbb{Z}/k\mathbb{Z} \rightarrow \mathbb{Z}_{\geq 0}$, defined by $m_{\tilde{A}_i}(\alpha) = \sum_{n \in \pi^{-1}(\alpha)} m_{A_{i}}(n) . $  The notation $A_1 = A_2 \mod k$ means that $\tilde{A}_1$ and $\tilde{A}_2$ are equal as multisets. 

\begin{lemma} \label{numberprop}
There does not exist three pairwise coprime integers $a_{1},a_{2},a_{3}$ such that $1 <a_1<a_2<a_3 $ and $\{a_1,a_2,a_3\} = \{-a_1,-a_2,-a_3\} \; \mod a_i,$ for $i=1,2,3$.
\end{lemma}
\begin{proof}
 By contradiction, let $1 < a_1 < a_2 < a_3$ be pairwise coprime with $\{a_1,a_2,a_3\} = \{-a_1,-a_2,-a_3\} \mod a_i$ for every $1 \le i \le 3$. Then $\{a_1, a_2\} = \{-a_1, -a_2\} \mod a_3$, hence either $a_1 = a_3-a_2$ or $a_1=a_3-a_1$. In the latter case, $a_3$ divides $2$ contradicting $3 < a_3$. Hence we have $a_1 = a_3-a_2$. Now, we have $\{a_1, a_3\} = \{-a_1, -a_3\} \mod a_2$, and so $\{a_3, a_3\} = \{-a_3, -a_3\} \mod a_2$ and in particular $a_3 = -a_3 \mod a_2$, implying that $a_2$ divides $2$, contradicting $a_2>2$.
\end{proof}

We are now ready to prove the main result about Hamiltonian $S^1$-actions on $8$-manifolds.

\begin{theorem} \label{eightmain} Let $(M,\omega)$ be a closed, connected, symplectic $8$-manifold with an effective Hamiltonian $S^1$-action such that all the non-zero weights along $M_{\min}$ and $M_{\max}$ have modulus greater than $1$. Then $b_{1}(M) \leq b_{3}(M)$.
\end{theorem}
\begin{proof}
By Lemma \ref{fourmin} it remains to prove the claim when $\dim(M_{\min}) = \dim(M_{\max}) =2$. By Theorem \ref{locbet} $b_1(M_{\min}) = b_1(M_{\max}) =b_1(M)$, so if  $b_1(M_{\min})=0$ the claim is proven. Therefore, assume additionally that $b_1(M_{\min})>0$. Consider an isotropy submanifold  $N$ \footnote{That is, $N$ is a component of the fixed point set of the action of $\mathbb{Z}_k \subset S^1$ for some $k \geq 2$.} containing $M_{\min}$, it will be shown that the desired inequality holds when $N$ is $6$-dimensional by a case analysis on the dimension of $N_{\max}$. 

If $N_{\max}$ is $4$-dimensional then by Lemma \ref{fixedintersection} and the fact that $\dim(M_{\max})=2$ there is non-extremal $4$-dimensional component $(N_{\max} =) F \subset M^{S^1}$. By applying Theorem \ref{locbet} to $N$ gives $b_{1}(F)=b_{3}(F)=b_{1}(M)$ and by a dimension count $\lambda_{F}=1$. The desired inequality can then be proven by Theorem \ref{locbet} to $b_{3}(M)$.

Next, we note that $N_{\max}$ cannot be isolated. This can be shown by applying Theorem \ref{locbet} to $N$ itself, since $$b_{1}(N_{\max}) =b_{1}(N) = b_{1}(N_{\min}) = b_{1}(M_{\min}) = b_{1}(M)>0,$$ it follows that $N_{\max}$ cannot be isolated. 

In the final remaining case that $N_{\max}$ is $2$-dimensional,  by Theorem \ref{sixfixed}  there exists a non-extremal fixed surface $F \subset N^{S^1} \subset M^{S^1}$ with a map $f : F \rightarrow N_{\min} = M_{\min}$ of degree $\deg(f) \in \{1,2\}$. In particular, $f_{*}:  H_{1}(F,\mathbb{R}) \rightarrow H_{1}(M_{\min},\mathbb{R})$ is surjective and it follows that $b_{1}(F) \geq b_{1}(M_{\min}) = b_{1}(M_{\max}) = b_{1}(M)$, and the claim is proven by Theorem \ref{locbet} and applying Poincar\'{e} duality to $M$.  By either reversing the circle action by Lemma \ref{reverse} or running the same argument with opposite signs, it follows that $M_{\max}$ cannot be contained in an isotropy $6$-manifold. Hence, it remains to prove the statement when all the isotropy submanifolds containing the extremal fixed surfaces $M_{\min}$ and $M_{\max}$ are $4$-dimensional.

Combining the conclusion of the previous paragraph with the assumption that non-zero weights along the extremal fixed submanifolds have modulus at least $1$, implies that the weights at $M_{\min},M_{\max}$ must be pairwise coprime because otherwise there is an isotropy $6$-manifold containing them.  Suppose there exists an isotropy $4$-manifold $N$ containing $M_{\min}$ but not $M_{\max}$. Then by Theorem \ref{locbet} applied to $N$, $F:=N_{\max} \subset M^{S^1}$ would give a non-extremal fixed surface of $M$ with $b_{1}(F) =b_{1}(M_{\min}) = b_{1}(M)$. Since $F$ is by assumption non-extremal, by a dimension count only $\lambda_{F}=1,2$ respectively are possible. By applying Theorem \ref{locbet} to $b_{3}(M),b_{5}(M)$ respectively, this gives that at least one of $b_{3}(M),b_{5}(M)$ is bounded below by $b_{1}(M)$.  The equality $b_{3}(M) = b_{5}(M)$ from Poincar\'{e} duality then implies the desired inequality.  Therefore, the statement is proven unless each of the three isotropy $4$-manifolds containing $M_{\min}$ contains $M_{\max}$. This is ruled out in the following paragraph.

Let $a,b,c>1$ be the weights at $M_{\min}$, as noted above they are pairwise coprime and in particular mutually distinct and so may be labeled $a_1,a_2,a_3$ so that $1<a_1<a_2<a_3$. As was shown above, each of the three corresponding isotropy $4$-manifolds $N_{i}$ containing $M_{\min}$ which are components of $M^{\mathbb{Z}_{a_i}}$ $i=1,2,3$ respectively contains $M_{\max}$. 

Because $M_{\max}$ has codimension $2$ in $N_{i}$ the weight of the action on $N_{i}$ at $M_{\max}$ is $-a_i$.  By assumption the weights at $M_{\max}$ are denoted $\{w_1,w_2,w_3\}, $ then picking a point $ q \in M_{\max}$ the action on $T_{q}M$ is $$z.(z_1,z_2,z_3) = (z^{w_1} z_1, z^{w_{2}} z_2, z^{w_{3}} z_3),$$ where by the assumption that non-zero weights along $M_{\max}$ have modulus greater than $1$, $w_i<-1$ for $i=1,2,3$. Moreover, as above since there is no isotropy $6$-manifold containing $M_{\max}$, it holds that $\gcd(w_i,w_j) = 1$ for $i \neq j$. 

The tangent spaces of the three isotropy $4$-manifolds correspond to three (real) dimension $2$ subspaces on which the stabilizer of a non-zero point is $\mathbb{Z}_{a_i}$ respectively. Therefore since the condition $\gcd(w_i,w_j) = 1$ for $i \neq j$ implies there are precisely three $2$-dimensional subspaces on which the stabilizer of every non-zero point is non-trivial, it follows that $\{w_1,w_2,w_3\} = \{-a_1,-a_2,-a_3\}$.

Finally, applying \cite[Lemma 2.6]{T} implies that $$\{a_1,a_2,a_3\} =  \{-a_1,-a_2,-a_3\} \;\;  \mod a_i,$$ for $i=1,2,3$. In Lemma \ref{numberprop} it was shown that this is not possible.
\end{proof}

Theorem \ref{eightmain} has the following consequence, proved by choosing a suitably generic linear subcircle $S^1 \subset T^2$.

\begin{corollary} \label{eitghtttwo}
Let $(M,\omega)$ be a closed, connected, symplectic $8$-manifold with an effective Hamiltonian $T^2$-action. Then $b_{1}(M) \leq b_{3}(M)$.
\end{corollary}

In the following section we give a proof of Corollary \ref{eitghtttwo} which is comparatively self-contained. However, Theorem \ref{eightmain} is substantially stronger than Corollary \ref{eitghtttwo}.

\section{Higher dimensional torus actions}
The main purpose of this section is to show Theorem \ref{tenmain}, showing that the first two odd Betti numbers of a closed $10$-dimensional manifold with an effective Hamiltonian $T^2$-action are unimodal.  Finally, using similar ideas we give an alternative comparatively self-contained proof of Corollary \ref{eitghtttwo} of the previous section.

\begin{theorem} \label{tenmain}
Let $(M,\omega)$ be a closed, connected, symplectic $10$-manifold with an effective Hamiltonian $T^2$-action. Then $b_{1}(M) \leq b_{3}(M)$.
\end{theorem}
\begin{proof}
Firstly, by Lemma \ref{fourmin} if $(M,\omega)$ has a Hamiltonian $S^1$-action with an extremal component of dimension $0$, $4$ or $8$ then the theorem is proved, this will be used in the following argument.

Let $\mu$ denote the moment map of the $T^2$-action. Consider an edge of $\mu(M)$, say $E$, and let $v_{1},v_{2}$ be the two vertices it contains. Then, by Lemma \ref{edgeiso} and Lemma \ref{vertexfixed} there are Hamiltonian $S^1$-actions on $(M,\omega)$ with extremal submanifolds equal to $\mu^{-1}(E),\mu^{-1}(v_{1}),\mu^{-1}(v_{2})$. Therefore, by Lemma \ref{fourmin} and the fact that $ \dim(\mu^{-1}(v_{i})) < \dim(\mu^{-1}(E) ) $ for $i=1,2$, the desired inequality is proved unless it holds that $\dim(\mu^{-1}(E) )= 6$ and $ \dim(\mu^{-1}(v_{i})) = 2$ for $i=1,2$.

By Lemma \ref{edgeiso}(3) the $6$-manifold  $N: = \dim(\mu^{-1}(E)) $ itself has a non-trivial Hamiltonian $S^1$-action. By Theorem \ref{six} the Betti numbers of $N$ satisfy the inequality $b_{1}(N) \leq b_{3} (N)$. Finally by by Lemma \ref{edgeiso}(2) there exists a Hamiltonian $S^1$-action with extremal component equal to $N$; by possible reversing it by  Lemma \ref{reverse} we can assume $N = M_{\min}$. Then Theorem \ref{locbet} gives $$b_{1}(M) =  \sum_{F \subset M^{S^1}} b_{1 -2\lambda_F} (F) = b_{1}(N)$$ and $$b_{3}(M)  =  \sum_{F \subset M^{S^1}} b_{3 -2\lambda_F} (F) =   b_{3}(N) + c,$$ where $c$ is a non-negative integer. Therefore $b_{1}(M) = b_{1}(N) \leq b_{3}(N) \leq b_{3}(M).$\end{proof}
Next, we give an alternative proof of Corollary \ref{eitghtttwo}.

\begin{proof}[Alternative proof of Corollary \ref{eitghtttwo}]

Firstly, by Lemma \ref{fourmin} if an $8$-dimensional closed symplectic manifold $(M,\omega)$ has a Hamiltonian $S^1$-action with an extremal component of dimension $0$, $4$ or $6$ then the corollary is proven, this will be used in the following argument.

Now assume that $(M,\omega)$ be an $8$-dimensional closed symplectic manifold with an effective Hamiltonian $T^2$-action with moment map $\mu$. Let $v$ be any vertex of the moment polygon $\mu(M)$, and $E_{1},E_{2}$ the adjacent edges. Then, by Lemma \ref{edgeiso} and Lemma  \ref{vertexfixed}, all three of the submanifolds $\mu^{-1}(v), \mu^{-1}(E_{1})$ and $ \mu^{-1}(E_{2})$ are extremal submanifolds for certain Hamiltonian $S^1$-actions on $M$. Therefore  by Lemma \ref{fourmin} if any of them are 0,4 or 6 dimensional the proof is complete. Finally since $\dim(\mu^{-1}(v)) < \dim(\mu^{-1}(E_{i}) ) <  8$ for $i=1,2$, at least one of them has dimension equal to $0,4$ or $6$, completing the proof.
\end{proof}

\subsection{A result in arbitrary dimensions}
In this subsection a result in arbitrary dimensions is proven. A complete solution of \cite[Problem 4.3]{JHKLM} would imply for a closed symplectic manifold with a non-trivial Hamiltonian $S^1$-action, for each odd $i$ such that $1 < i \leq n$, $$b_1(M) \leq b_3(M) \leq \ldots \leq b_i(M).$$ In Theorem \ref{gendim}, we show that in arbitrary dimensions with Hamiltonian $T^2$-symmetry, this inequality holds for \textit{some} $i$. The following lemma about Betti numbers will be needed.

\begin{lemma} \label{oddgen}
Suppose that $(M,\omega)$ is a $2n$-dimensional, closed, connected, symplectic manifold with a non-trivial Hamiltonian $S^1$-action, and has an extremal fixed component $F_{0}$ of dimension $2k \neq 2$. Then there is an odd integer $i$ such that $1<i \leq n$ and $b_{1}(M) \leq b_{i}(M)$.
\end{lemma}
\begin{proof}
If $k=0$ then $b_{1}(M)=0$ by Theorem \ref{locbet}. If $2< 2k <2n$, then by Lemma \ref{unimod} the inequalities $b_{1}(M) \leq b_{2k-1}(M)$ and $b_{1}(M) \leq b_{2n-2k+1}(M)$ hold. The assumption $2k>2$ implies that the two odd integers $2k-1$ and $2n-2k+1$ are greater than $1$. Moreover since they differ by a reflection centered at $n$, at least one of them is at most $n$. 
\end{proof}

The main result of this subsection is as follows.

\begin{theorem} \label{gendim}
Let $(M,\omega)$ be a closed, connected, symplectic $2n$-dimensional manifold with an effective Hamiltonian $T^2$-action, then there is an odd integer $i$ such that $1<i \leq n$ and  $b_{1}(M) \leq b_{i}(M)$.
\end{theorem}
\begin{proof}
By Lemma \ref{edgeiso} and Lemma  \ref{vertexfixed} the moment map preimages of all vertices and edges of the moment polygon are extremal submanifolds for certain Hamiltonian $S^1$-actions on $M$. Since the preimage of an edge has greater dimension than the preimages of the vertices it contains, at least one of these submanifolds does not have dimension $2$. Therefore the proof follows by Lemma \ref{oddgen}.
\end{proof}

Department of Mathematics Education, Sungkyunkwan University, Seoul, Republic of Korea.\\

\noindent Email: 20260280@skku.edu.


\begin{thebibliography}{9}

\bibitem{A} M. F. Atiyah. Convexity and Commuting Hamiltonians. Bulletin of the London Mathematical Society. Volume 14, Issue 1, pp. 1-15.

\bibitem{C} Y. Cho. Unimodality of Betti numbers for Hamiltonian circle actions with index increasing moment maps. International J. Math. 27. No. 5. 1650043 (14 pages).

\bibitem{C1} Y. Cho. Hard Lefchetz property of symplectic structures on compact Kaehler manifolds. Trans. Amer. Math. Soc. 368, No. 11, pp. 8223-8248.

\bibitem{CK} Y. Cho, M. Kim. Unimodality of the Betti numbers for Hamiltonian circle actions with isolated fixed points. Math. Res. Lett. 21, No.4. 691-696. 

\bibitem{CK3} Y. Cho, M. Kim. Hard Lefschetz property for Hamiltonian torus actions on 6-dimensional GKM-manifolds. J. Symplectic Geom. 16 (2018), No. 6, 1549-1590.



\bibitem{GS} V. Guillemin, S. Sternberg. Convexity Properties of the Moment Mapping. Inventiones Mathematicae. Volume 67, pp. 491-514. 

\bibitem{JHKLM} L. Jeffrey, T. Holm, Y. Karshon, E. Lerman, and E. Meinrenken, Moment maps in various geometries, available at
http://www.birs.ca/workshops/2005/05w5072/report05w5072.pdf.


\bibitem{Ka} Y. Karshon. Periodic Hamiltonian flows on four dimensional manifolds, Memoirs
 Amer. Math. Soc. 672, 71p, (1999).

\bibitem{Ki} F. C. Kirwan. The cohomology of quotients in symplectic and algebraic geometry, Princeton University Press, 1984.


\bibitem{Le} S. Lefschetz. L'Analysis situs et la géométrie algébrique. Gauthier-Villars, Paris. 



\bibitem{L2} N. Lindsay. Cohomological Localization for Hamiltonian $S^1$-actions and symmetries of complete intersections. Journal of the London Mathematical Society. Volume 112, Issue 5, e70359.



\bibitem{LP1} 
N. Lindsay, D. Panov.  Invariant symplectic hypersurfaces in dimension $6$ and the Fano condition. Journal of Topology. Volume12, Issue1 (2019).Pages 221-285.




\bibitem{MS} D. McDuff, D. Salamon. Introduction to Symplectic Topology, Second Edition. Oxford Mathematical Monographs. 


\bibitem{R} Y. Rong, Degree one maps between geometric $3$-manifolds. Trans. Amer. Math. Soc. 332 (1992), no. 1, 411--436.

\bibitem{T} S. Tolman. On a symplectic generalization of petrie’s conjecture. Trans. Amer. Math. Soc.,
362(8):3963–3996, 2010.

\bibitem{TW} S. Tolman. J. Weitsman. On semifree symplectic circle actions with isolated fixed points. Topology. Volume 39, Issue 2, pp. 299-309.

\bibitem{Th} W.P. Thurston. Some Simple Examples of Symplectic Manifolds. Proceedings of the American Mathematical Society, Vol. 55, No. 2 (Mar 1976), pp. 467-468.
\end{thebibliography}
\end{document}